\documentclass[12pt,oneside,reqno]{amsart}
\allowdisplaybreaks
\usepackage[top=3cm, bottom=3cm, left=3cm, right=3cm]{geometry}
\usepackage{multirow,tabularx}
\newcolumntype{Y}{>{\centering\arraybackslash}X}

\usepackage{pbox}
\usepackage{parskip}
\usepackage{amssymb}
\usepackage{hyperref}
\usepackage{mathtools}
\usepackage{enumerate}
\usepackage{tikz-cd}
\usepackage[cmtip,all]{xy}
\newcommand{\longsquiggly}{\xymatrix{{}\ar@{~>}[r]&{}}}

\newtheorem{thm}{Theorem}[section]

\newtheorem{lemma}[thm]{Lemma}

\makeatletter
\newcommand{\BIG}{\bBigg@{2}}
\newcommand{\vast}{\bBigg@{3}}
\newcommand{\Vast}{\bBigg@{5}}
\makeatother

\numberwithin{equation}{section}

\begin{document}

%%%%% To ease editing, for IMPAN journals add:

\baselineskip=17pt

%%%%%%%%%%%

%% In the running head, replace first names by initials
%% and give an abbreviation of the title.

\title[$(\sigma,\tau)$-Derivations on Commutative Algebras]{A Note on $\boldsymbol{(\sigma,\tau)}$-Derivations on Commutative Algebras}
\author[D.Chaudhuri]{Dishari Chaudhuri}
\address{Dishari Chaudhuri\\Department of Mathematical Sciences\\ Indian Institute of Science Education and Research Mohali\\
Sector-81, Knowledge City, S.A.S. Nagar, Mohali-140306\\ Punjab, India}
\email{dishari@iitg.ac.in, dishari.chaudhuri@gmail.com}

%\date\today

\begin{abstract}We study universal mapping properties of $(\sigma,\tau)$-derivations over commutative algebras and characterize them over rings of integers of quadratic number fields. As a result we provide extension of some well known results on UFD's of such derivations to certain non-UFD's as well.\\

 \noindent Keywords: $(\sigma,\tau)$-derivations, universal properties, ring of integers.\\

 \noindent $2010$ Mathematics Subject Classification: $16$W$25$, $16$S$10$, $11$R$04$.

 %We also give an interesting observation on the derived length of a group algebra of a group of odd order over fields of certain characteristics.
\end{abstract}

%\keywords{Group algebras, units, derived subgroups}
%
%\subjclass[2010]{Primary 16S34; Secondary 16U60 }

\maketitle

%%%%%%%%%%%%%%%%%%%%%%%%%%%%%%%%%%SECTION 1%%%%%%%%%%%%%%%%%%%%%%%%%%%%%%%%%%%%%%%%%%%%%%%%%%%%%%%
\section{Introduction}
Let $R$ be a commutative ring with $1$ and $\mathcal{A}$ be an algebra over $R$.
%A derivation on $\mathcal{A}$ is an $R$-linear map $\gamma: \mathcal{A}\rightarrow\mathcal{A}$ satisfying $\gamma(ab)=\gamma(a)b+a\gamma(b)$ for all $a,b\in\mathcal{A}$. For $x\in\mathcal{A}$, the derivation $\gamma$ such that $\gamma(a)=xa-ax$ for all $a\in\mathcal{A}$ is called an inner derivation of $\mathcal{A}$ coming from $x$. We will denote such a derivation and inner derivation as $\gamma^{usual}$ and $\gamma_x^{inner}$ respectively.
Let $\sigma$, $\tau$ be two different algebra endomorphisms on $\mathcal{A}$. A $(\sigma,\tau)$-derivation on $\mathcal{A}$ is an $R$-linear map $\delta: \mathcal{A}\longrightarrow\mathcal{A}$ satisfying $\delta(ab)\;=\; \delta(a)\tau(b)+\sigma(a)\delta(b)$ for $a,b\in \mathcal{A}$. If $x\in\mathcal{A}$, the $(\sigma,\tau)$-derivation $\delta_x$ such that $\delta_{x}(a)=x\tau(a)-\sigma(a)x$, is called a $(\sigma,\tau)$-inner derivation of $\mathcal{A}$ coming from $x$. If $\sigma=\tau=id$, then $\delta$ and $\delta_x$ are respectively the ordinary derivation and inner derivation of $\mathcal{A}$ coming from $x$. For given $\sigma$ and $\tau$ on $\mathcal{A}$, the set of all $(\sigma,\tau)$-derivations on $\mathcal{A}$ is denoted by $\mathfrak{D}_{(\sigma,\tau)}(\mathcal{A})$ and whenever $\mathcal{A}$ is commutative, it carries a natural left module structure given by $(a,\delta)\longmapsto a\cdot\delta:\;x\mapsto a\delta(x).$

The use of $(\sigma,\tau)$-derivations in generalizing Galois theory over division rings were mentioned by Jacobson in \cite{J} (Chapter $7.7$). Later on they have been studied extensively in the case of prime and semiprime rings by many authors. \cite{AAH-06} contains a survey of such studies. An introductory exposure to $(\sigma,\tau)$-derivations can be found in \cite{H-02}. For examples of these kind of twisted derivations on UFD's one can refer to (Table $1$, \cite{ELMS-16} and Table $1$, \cite{H-02}). In the last decade Hartwig, Larsson and Silvestrov gave the study of twisted derivations a new dimension when they generalized Lie algebras to Hom-Lie algebras using $(\sigma,\tau)$-derivations on associative algebras $\mathcal{A}$ over $\mathbb{C}$ in their highly influential paper \cite{HLS-06}. Just as Lie algebras were initially studied as algebras of derivations, hom-Lie algebras have been studied as algebras of twisted derivations. Following this,  twisted derivations and hom-Lie algebras have been an interesting object of study for various different contexts. For example, \cite{G-10}, \cite{JL-08}, \cite{MM-09}, \cite{S-12} to name a few. 
Hartwig et al. mostly studied $(\sigma,\tau)$-derivations on UFD's. An important result (Theorem $4$) in \cite{HLS-06} states that if $\sigma$ and $\tau$ are two different algebra endomorphisms on a unique factorization domain $\mathcal{A}$, then $\mathcal{D}_{(\sigma,\tau)}(\mathcal{A})$ is free of rank one as an $\mathcal{A}$-module with generator $$\Delta:=\frac{\tau-\sigma}{g}:\;x\longmapsto\frac{(\tau-\sigma)(x)}{g},\text{ where }g=\;gcd\;((\tau-\sigma)(\mathcal{A})).$$

  All these results motivated us to investigate similar properties of $(\sigma,\tau)$-derivations in a non unique factorization domain as gcd may not exist for such cases or in algebras in general. It turns out that in quite a few cases $(\sigma,\tau)$-derivations are actually  $(\sigma,\tau)$-inner. We have found that finite dimensional central simple algebras have that property and group rings of finite groups over integral domains with $1$ also share the same property for certain $\sigma$ and $\tau$ (\cite{Chau-19}). Also it is interesting to investigate whether for fixed $\sigma$ and $\tau$, there exists a module of $(\sigma,\tau)$-derivations satisfying universal properties analogous to that of K\"{a}hler modules in case of ordinary derivations. In this note we discuss the universal mapping properties of $(\sigma,\tau)$ derivations on commutative algebras (Theorem \ref{main}). We also characterize them over rings of integers of quadratic number fields (Theorem \ref{O_2}), thus providing examples of such derivations over some non-UFD's as well.
 
 \section{Definition}

We generalize the definition of $(\sigma,\tau)$-derivations on an algebra to $(\sigma,\tau)$-derivations from an algebra to a bimodule over that algebra. In the sequel all algebras considered will be commutative. Let $\mathcal{A}$, $\mathcal{B}$ and $\mathcal{C}$ be algebras over $R$.  Let $\sigma$ be an algebra homomorphism from $\mathcal{A}$ to $\mathcal{B}$ and $\tau$ be an algebra homomorphism from $\mathcal{A}$ to $\mathcal{C}$. Let $M$ be a $\mathcal{B}-\mathcal{C}$-bimodule. $\mathcal{A}$ is naturally an $\mathcal{A}-\mathcal{A}$-bimodule. Now $M$ is an $\mathcal{A}-\mathcal{A}$-bimodule too via $\sigma$ and $\tau$. We will assume  $\mathcal{A}$ acts on $M$ from the left via $\sigma$ and from the right via $\tau$. A $(\sigma,\tau)$-derivation $D$ from $\mathcal{A}$ to $M$ is an $R$-linear map $D:\mathcal{A}\longrightarrow M$ satisfying $D(ab)= D(a)\cdot\tau(b)+\sigma(a)\cdot D(b)$ for $a,b\in \mathcal{A}$.  If $x\in\mathcal{A}$, the $(\sigma,\tau)$-derivation $D_x:\mathcal{A}\rightarrow M$ such that $D_{x}(a)=x\cdot\tau(a)-\sigma(a)\cdot x $, is called a $(\sigma,\tau)$-inner derivation of $\mathcal{A}$ to $M$ coming from $x$. If $\mathcal{B}=\mathcal{C}=\mathcal{A}$ and $\sigma=\tau=id$, then $D$ and $D_x$ are respectively the ordinary derivation and inner derivation from $\mathcal{A}$ to the $\mathcal{A}-\mathcal{A}$-bimodule $M$. For a given $\sigma$ and $\tau$, let us denote the set of all $(\sigma,\tau)$-derivations from $\mathcal{A}$ to $M$ as $\mathfrak{D}_{(\sigma,\tau)}(\mathcal{A}, M)$. Then $\mathfrak{D}_{(\sigma,\tau)}(\mathcal{A},M)$ carries a natural $\mathcal{A}-\mathcal{A}$-bimodule structure by $\big((a_1,a_2),\delta\big)\longmapsto (a_1, a_2)\cdot\delta:\;x\mapsto \sigma(a_1)\cdot\delta(x)\cdot \tau(a_2).$ Also note that $D(1)=0.$

\section{Universal properties of $(\sigma,\tau)$-derivations on commutative algebras}

 We are now in a position to discuss universal mapping properties of $(\sigma,\tau)$-derivations. In the sequel all tensor products will be taken over $R$.

 \begin{thm}\label{main}Let $R$ be a commutative ring with $1$, $\mathcal{A}$ be an algebra over $R$ and $M$ be an $\mathcal{A}-\mathcal{A}$-bimodule. Let $\sigma$, $\tau$ be two different algebra endomorphisms on $\mathcal{A}$ and $D$ be a $(\sigma,\tau)$-derivation from $\mathcal{A}$ to $M$. Let $K_\sigma,\;K_\tau$ and $K_D$ denote the kernels of $\sigma,\;\tau$ and $D$ respectively. Let $\phi_\sigma$ denote the algebra isomorphism from $\mathcal{A}/K_\sigma$ to $\sigma(\mathcal{A})$ and  $\psi_\tau$ denote the algebra isomorphism from $\mathcal{A}/K_\tau$ to $\tau(\mathcal{A})$. Also let $\pi_\tau$ and $\pi_\sigma$ denote the projection homomorphisms from $\mathcal{A}$ to $\mathcal{A}/K_\tau$ and $\mathcal{A}$ to $\mathcal{A}/K_\sigma$ respectively. Then depending on the invertibility of $\sigma$ and $\tau$ the following universal mapping properties of $(\sigma,\tau)$-derivations are satisfied:
 	\begin{enumerate}[$(1)$]
 		\item If $\tau$ is invertible with $\sigma$ may or may not be invertible, then for any $D\in\mathfrak{D}_{(\sigma,\tau)}(\mathcal{A},M)$, there exists an $\mathcal{A}-\mathcal{A}$-bimodule $J$,  $\delta_L\in\mathfrak{D}_{(\sigma,\tau)}(\mathcal{A},J)$ and a unique left $\mathcal{A}$-linear map $f_L:J\longrightarrow M$ such that $D=f_L\circ\delta_L.$
 		\item If $\tau$ is not invertible but $\sigma$ is invertible, then for any $D\in\mathfrak{D}_{(\sigma,\tau)}(\mathcal{A},M)$, there exists an $\mathcal{A}-\mathcal{A}$-bimodule $J$,  $\delta_R\in\mathfrak{D}_{(\sigma,\tau)}(\mathcal{A},J)$ and a unique right $\mathcal{A}$-linear map $f_R:J\longrightarrow M$ such that $D=f_R\circ\delta_R.$
 		\item If neither $\tau$ nor $\sigma$ is invertible, then for any $D\in\mathfrak{D}_{(\sigma,\tau)}(\mathcal{A},M)$ with $K_\tau\subseteq K_D$, there exists an $\mathcal{A}-\mathcal{A}$-bimodule $\bar{J}$,  $\Delta_L\in\mathfrak{D}_{(\sigma,\psi_\tau\circ\pi_\tau)}(\mathcal{A},\bar{J})$ and a unique left $\mathcal{A}$-linear map $g_L:\bar{J}\longrightarrow M$ such that $D=g_L\circ\Delta_L.$
 		\item If neither $\tau$ nor $\sigma$ is invertible, then for any $D\in\mathfrak{D}_{(\sigma,\tau)}(\mathcal{A},M)$ with $K_\sigma\subseteq K_D$, there exists an $\mathcal{A}-\mathcal{A}$-bimodule $\tilde{J}$, $\Delta_R\in\mathfrak{D}_{(\phi_\sigma\circ\pi_\sigma,\tau)}(\mathcal{A},\tilde{J})$ and a unique right $\mathcal{A}$-linear map $g_R:\tilde{J}\longrightarrow M$ such that $D=g_R\circ\Delta_R.$
 	\end{enumerate}

 	\end{thm}
 
 \begin{proof}For each case we need to show that the following diagram is commutative:
 	\begin{eqnarray}\label{diagram}
 	\begin{tikzcd}
 	{} \mathcal{A}\arrow{dd}{(\sigma,\tau)}[swap]{D}\arrow{rr}{ \delta_0}[swap]{(\Sigma,T)} & & M_0 \arrow{lldd}{\exists\; !\; f}\\
 	& &\\
 	M & &
 	\end{tikzcd}
 	\end{eqnarray}
 	 where $D$ is a $(\sigma,\tau)$-derivation from $\mathcal{A}$ to $M$, $\delta_0$ is a $(\Sigma,T)$-derivation from $\mathcal{A}$ to $M_0$ and $M_0,\;\Sigma,\;T,\;\delta_0$ and $f$ for each case are as follows:
 	 \begin{enumerate}[C\text{a}se $(1)$:]
 	 	\item $M_0=J$, $\Sigma=\sigma,\;T=\tau,\;\delta_0=\delta_L$ and $f=f_L$,
 	 	\item $M_0=J$, $\Sigma=\sigma,\;T=\tau,\;\delta_0=\delta_R$ and $f=f_R$,
 	 	\item $M_0=\bar{J}$, $\Sigma=\sigma,\;T=\psi_\tau\circ\pi_\tau,\;\delta_0=\Delta_L$ and $f=g_L$,
 	 	\item $M_0=\tilde{J}$, $\Sigma=\phi_\sigma\circ\pi_\sigma,\;T=\tau,\;\delta_0=\Delta_L$ and $f=g_R$.\\
 	 	\end{enumerate}
  	
  	That is, we need to show that for any given $(\sigma,\tau)$-derivation $D$ from $\mathcal{A}$ to $M$, there exist algebras $\mathcal{B},\;\mathcal{C}$ over $R$, algebra homomorphisms $\Sigma$ from $\mathcal{A}$ to $\mathcal{B}$ and $T$ from $\mathcal{A}$ to $\mathcal{C}$, a $\mathcal{B}-\mathcal{C}$-bimodule $M_0$ which is an $\mathcal{A}-\mathcal{A}$-bimodule as well via $\Sigma$ from the left and $T$ from the right such that there exists a unique left or right $\mathcal{A}$-linear map $f$ from $M_0$ to $M$ making the above diagram commutative.
    
We now prove each case separately.\\

 \underline{Case $(1)$:} $\tau$ is invertible, $\sigma$ may or may not be invertible.\\
 We view $\mathcal{A}$ as an embedding in $\mathcal{A}\otimes \mathcal{A}$ via the first component. That is, we regard $\mathcal{A}\otimes\mathcal{A}$ as a left $\mathcal{A}$-module via the first component. Let $J$ be the two sided ideal of $\mathcal{A}\otimes\mathcal{A}$ generated by elements of the form $1\otimes\tau(a)-\sigma(a)\otimes 1.$ We take $J$ as an $\mathcal{A}-\mathcal{A}$ bimodule via $\sigma$ from the left and $\tau$ on the right. We define $\delta_L : \mathcal{A}\longrightarrow J$ by $\delta_L(a)=1\otimes\tau(a)-\sigma(a)\otimes 1$. Then clearly $\delta_L$ is $R$-linear and
 \begin{flalign*}
\delta_L(ab)&=1\otimes\tau(ab)-\sigma(ab)\otimes1=1\otimes\tau(a)\tau(b)-\sigma(a)\sigma(b)\otimes1\\
            &=\big(1\otimes\tau(a)\big)\big(1\otimes\tau(b)\big)-\big(\sigma(a)\otimes1)\big)\big(\sigma(b)\otimes1)\big)\\
            &=\big(1\otimes\tau(a)-\sigma(a)\otimes1\big)\big(1\otimes\tau(b)\big)+\big(\sigma(a)\otimes1\big)\big(1\otimes\tau(b)-\sigma(b)\otimes1\big)\\
            &=\delta_L(a)\cdot\tau(b)+\sigma(a)\cdot\delta_L(b).
\end{flalign*}

Thus we see that $\delta_L$ is a $(\sigma,\tau)$-derivation. Now, we define $f_L: J\longrightarrow M$ to be the $J$-restriction of the map from $\mathcal{A}\otimes\mathcal{A}$ to $M$ given by $f(x\otimes y)=xD\big(\tau^{-1}(y)\big)$. Then $f_L$ is an $R$-bilinear map and is a left $\mathcal{A}$-linear map too by our definition of the left $\mathcal{A}$-module structure on $\mathcal{A}\otimes\mathcal{A}$. As $f_L$ is defined on the generators of $J$, it is unique too. Also $(f_L\circ\delta_L)(a)=D\big(\tau^{-1}(\tau(a))\big)-\sigma(a)D\big(\tau^{-1}(1)\big)=D(a)$ as $D(1)=0$.\\

\underline{Case $(2)$:} $\tau$ is not invertible, but $\sigma$ is invertible. \\
Here we assume $\mathcal{A}\otimes\mathcal{A}$ is a right $\mathcal{A}$-module via the second component. Let us define an $\mathcal{A}-\mathcal{A}$-bimodule $J$ via $\sigma$ and $\tau$ in the same way as case $(1)$. Then we define $\delta_R$ from $\mathcal{A}$ to $J$ again in the same way as case $(1)$ and similarly it is a $(\sigma,\tau)$-derivation. Now we define $f_R: J\longrightarrow M$ to be the $J$-restriction of the map from $\mathcal{A}\otimes\mathcal{A}$ to $M$ given by $f_R(x\otimes y)=-D\big(\sigma^{-1}(x)\big)y.$ Then similarly $f_R$ is a unique $R$-bilinear map and is a right $\mathcal{A}$-linear map too by our definition of the right $\mathcal{A}$-module structure on $\mathcal{A}\otimes\mathcal{A}$. Also as $D(1)=0$, we get as before $f_R\circ\delta_R=D.$\\

\underline{Case $(3)$:} Neither $\tau$ nor $\sigma$ is invertible and we need to prove universal module and map exist only for those $(\sigma,\tau)$-derivations $D$ such that the kernel of $\tau$ is contained in the kernel of $D$, that is, $D(K_\tau)=0$.\\
We view $\mathcal{A}\otimes\mathcal{A}/K_\tau$ as a left $\mathcal{A}$-module via the first component. We define the map $D_\tau: \mathcal{A}\longrightarrow M$ by $D(a+K_\tau)=D(a)$ for all $a\in \mathcal{A}.$ As $D(K_\tau)=0$, the map $D_\tau$ is well defined. Let $\bar{J}$ be the two-sided ideal of $\mathcal{A}\otimes\mathcal{A}/K_\tau$ generated by the elements $1\otimes(\psi_\tau\circ\pi_\tau)(a)-\sigma(a)\otimes(1+K_\tau)$. So $\bar{J}$ is naturally an $\mathcal{A}-\mathcal{A}/K_\tau$-bimodule. It is also an $\mathcal{A}-\mathcal{A}$-bimodule via $\sigma$ on the left and $\psi_\tau\circ\pi_\tau$ on the right. Now we define $\Delta_L:\mathcal{A}\longrightarrow \bar{J}$ by $\Delta_L(a)=1\otimes(\psi_\tau\circ\pi_\tau)(a)-\sigma(a)\otimes(1+K_\tau).$ Then we can see that $\Delta_L$ is a $(\sigma,\psi_\tau\circ\pi_\tau)$-derivation from $\mathcal{A}$ to $\bar{J}$ because of the following:

\begin{flalign*}
\Delta_L(ab)&=1\otimes(\psi_\tau\circ\pi_\tau)(ab)-\sigma(ab)\otimes(1+K_\tau)=1\otimes\psi_\tau(ab+K_\tau)-\sigma(ab)\otimes(1+K_\tau)\\
&=1\otimes\psi_\tau(a+K_\tau)\psi_\tau(b+K_\tau)-\sigma(a)\sigma(b)\otimes(1+K_\tau)\\
&=\big(1\otimes\psi_\tau(a+K_\tau)\big)\big(1\otimes\psi_\tau(b+K_\tau)\big)-\big(\sigma(a)\otimes(1+K_\tau)\big)\big(\sigma(b)\otimes(1+K_\tau)\big)\\
&=\big(1\otimes\psi_\tau(a+K_\tau)-\sigma(a)\otimes(1+K_\tau)\big)\big(1\otimes\psi_\tau(b+K_\tau)\big)\\
&\qquad+\big(\sigma(a)\otimes(1+K_\tau)\big)\big(1\otimes\psi_\tau(b+K_\tau)-\sigma(b)\otimes(1+K_\tau)\big)\\
&=\delta_L(a)\cdot(\psi_\tau\circ\pi_\tau)(b)+\sigma(a)\cdot\delta_L(b).
\end{flalign*}

 Finally, we define $g_L: \bar{J}\longrightarrow M$ to be the $\bar{J}$-restriction of the map from $\mathcal{A}\otimes\mathcal{A}/K_\tau$ to $M$ given by $g_L\big(x\otimes (y+K_\tau)\big)=xD_\tau\big(\psi_\tau^{-1}(y+K_\tau)\big)$. Then $g_L$ is a unique $R$-bilinear map as it is defined on the generators of $\bar{J}$ and left $\mathcal{A}$-linear too by our assumption. Also $(g_L\circ\Delta_L)(a)=g_L\big(1\otimes\psi_\tau(a+K_\tau)-\sigma(a)\otimes(1+K_\tau)\big)=D_\tau(a+K_\tau)=D(a)$.\\

\underline{Case $(4)$:} Neither $\tau$ nor $\sigma$ is invertible and we need to prove universal module and map exist only for those $(\sigma,\tau)$-derivations $D$ such that the kernel of $\sigma$ is contained in the kernel of $D$, that is, $D(K_\sigma)=0$.\\
We view $\mathcal{A}/K_\sigma\otimes\mathcal{A}$ as a right $\mathcal{A}$-module via the second component. We define the map $D_\sigma: \mathcal{A}\longrightarrow M$ by $D(a+K_\sigma)=D(a)$ for all $a\in \mathcal{A}.$ As $D(K_\sigma)=0$, the map $D_\sigma$ is well defined. Let $\tilde{J}$ be the two-sided ideal of $\mathcal{A}/K_\sigma\otimes\mathcal{A}$ generated by the elements $(1+K_\sigma)\otimes\tau(a)-(\phi_\sigma\circ\pi_\sigma)(a)\otimes1$. We have $\tilde{J}$ is an $\mathcal{A}/K_\sigma-\mathcal{A}$-bimodule which is an $\mathcal{A}-\mathcal{A}$-bimodule too via $\phi_\sigma\circ\pi_\sigma$ on the left and $\tau$ on the right. Now we define $\Delta_R:\mathcal{A}\longrightarrow \tilde{J}$ by $\Delta_R(a)=(1+K_\sigma)\otimes\tau(a)-(\phi_\sigma\circ\pi_\sigma)(a)\otimes1.$ Then as above we can see that $\Delta_R$ is a $(\phi_\sigma\circ\pi_\sigma,\tau)$-derivation from $\mathcal{A}$ to $\tilde{J}.$ Finally, we define $g_R: \mathcal{A}/K_\sigma\otimes\mathcal{A}\longrightarrow M$ to be the $\tilde{J}$-restriction of the map from $\mathcal{A}/K_\sigma\otimes\mathcal{A}$ to $M$ given by $g_R\big((x+K_\sigma)\otimes y)=-D_\sigma\big(\phi_\sigma^{-1}(x+K_\sigma)\big)y$. Then similarly $g_R$ is a unique $R$-bilinear map and right $\mathcal{A}$-linear too by our assumption. Also $g_R\circ\Delta_R=D$.

\end{proof}

%%%%%%%%%%%%%%%%%%%%%%%%%%%%%%%%%%%%%%%%%%%%%%%%%%%%%%%%%%%%%%%%%%

\section{A study of $(\sigma,\tau)$-derivations on rings of integers of quadratic number fields}

We now investigate $(\sigma,\tau)$-derivations on rings of integers of number fields. We will assume that $K$ is a number field and $\mathcal{O}_K$ is its ring of integers which we will consider as $\mathbb{Z}$-algebras. Thus, $\sigma,\;\tau$ will be ring endomorphisms on $\mathcal{O}_K$ and $(\sigma,\tau)$-derivations on $\mathcal{O}_K$ will be $\mathbb{Z}$-linear. We start with the following observation.

\begin{lemma}\label{O_k}
Let $K$ be a number field and $\mathcal{O}_K$ be its ring of integers. Any non-zero ring endomorphism of $\mathcal{O}_K$ is an isomorphism.
\end{lemma}

\begin{proof} Let $\phi$ be a non-zero ring endomorphism of $\mathcal{O}_K$. Now, $\mathcal{O}_K$ is a free $\mathbb{Z}$-module of finite rank and $\mathcal{O}_K/Ker \;\phi\cong Im\; \phi$ which is a $\mathbb{Z}$-submodule of $\mathcal{O}_K$, hence a free $\mathbb{Z}$-module. It is well known that $|\mathcal{O}_K/Ker \;\phi|<\infty$ unless $Ker\; \phi=0$ (see e.g., the proof of Theorem $5.3(d)$, \cite{ST} ). This implies rank of $Im\;\phi=0$ which in turn implies that $\phi$ is a trivial endomorphism and that is a contradiction to our assumption. Therefore we must have $Ker\;\phi=0$, that is, $\phi$ is injective. \par
Let $x\in\mathcal{O}_K$. Let $C_x=\{y\in \mathcal{O}_K\;|\;y\;\text{is conjugate to }x\},$ that is, $C_x$ is the set of all roots of the monic polynomial corresponding to $x$. Now since $\phi$ restricted to $\mathbb{Z}$ is the identity map, we have $\phi(C_x)\subseteq C_x.$ Now for $x_i,\;x_j\in C_x$ such that $x_i\neq x_j$, as $\phi$ is injective we will have $\phi(x_i)\neq\phi(x_j)$. Then by pigeonhole principle, there exists $x_i\in C_x$ such that $\phi(x_i)=x$. Thus $\phi$ is surjective as well.
\end{proof}

Now we are mainly interested in the rings of integers of quadratic number fields. Let $K$ be a number field such that $[K:\mathbb{Q}]=2$. Then we know that for a square free rational integer $d$, 
$$ \mathcal{O}_K=\begin{cases}
    \mathbb{Z}[\sqrt{d}], & \text{if $d\not\equiv 1$ $(\text{mod }4)$},\\
    \mathbb{Z}[\frac{1}{2}+\frac{\sqrt{d}}{2}], & \text{if $d\equiv 1$ $(\text{mod }4)$}.
  \end{cases}$$
  Our result regarding $(\sigma,\tau)$-derivations on such rings is as follows:
  
  \begin{thm}\label{O_2}Let $K$ be a number field such that $[K:\mathbb{Q}]=2$ and $\mathcal{O}_K$ be its ring of integers. Let $\sigma$ and $\tau$ be two different non-zero ring endomorphisms of $\mathcal{O}_K$ and $d$ be a square free integer.
  \begin{enumerate}[$(i)$]
   \item If $d\not\equiv 1$ mod $4$, then any $\mathbb{Z}$-linear map $D:\mathbb{Z}[\sqrt{d}]\longrightarrow\mathbb{Z}[\sqrt{d}]$ such that $D(1)=0$ is a $(\sigma,\tau)$-derivation of $\mathbb{Z}[\sqrt{d}]$. Moreover, if $D(\sqrt{d})=\alpha+\beta\sqrt{d}$ for some $\alpha,\beta\in\mathbb{Z}$ such that $\alpha$ is divisible by $2d$ and $\beta$ is even, then $D$ is $(\sigma,\tau)$-inner.
   \item If $d\equiv 1$ mod $4$, then any $\mathbb{Z}$-linear map $D:\mathbb{Z}[\frac{1+\sqrt{d}}{2}]\longrightarrow\mathbb{Z}[\frac{1+\sqrt{d}}{2}]$ such that $D(1)=0$ is a $(\sigma,\tau)$-derivation of $\mathbb{Z}[\frac{1+\sqrt{d}}{2}]$. Moreover, if $D\big(\frac{1+\sqrt{d}}{2}\big)=\alpha+\beta\big(\frac{1+\sqrt{d}}{2}\big)$ for some $\alpha,\;\beta\in\mathbb{Z}$ such that $\beta$ even, then $D$ is $(\sigma,\tau)$-inner.
   \end{enumerate}
   
   \end{thm}
   
   \begin{proof}\begin{enumerate}[$(i)$]
   \item Here, $d\not\equiv 1$ mod $4$. By lemma \ref{O_k}, we know that $\sigma$ and $\tau$ are ring isomorphisms. We have $d=\sigma(d)=\sigma(\sqrt{d})^2$ and so $\sigma(\sqrt{d})=\pm\sqrt{d}$. Similarly, $\tau(\sqrt{d})=\pm\sqrt{d}$. As $\sigma$ and $\tau$ are different, clearly we have only two choices for the pair $(\sigma,\tau)$. For all $a,b\in\mathbb{Z}$, they are: \begin{enumerate}[$(a)$]
                                                                                                                                                                                                                                                                            \item $\sigma=id;\quad\tau(a+b\sqrt{d})=a-b\sqrt{d}$,
                                                                                                                                                                                                                                                                            \item $\sigma(a+b\sqrt{d})=a-b\sqrt{d};\quad\tau=id$.
                                                                                                                                                                                                                                                                            
                                                                                                                                                                                                                                                                           \end{enumerate}
Let us assume case $(a)$. Case $(b)$ will follow from symmetry. Let $D:\mathbb{Z}[\sqrt{d}]\longrightarrow\mathbb{Z}[\sqrt{d}]$ be a $\mathbb{Z}$-linear map such that $D(1)=0$. Let $D(\sqrt{d})=\alpha+\beta\sqrt{d}$ for some $\alpha,\;\beta\in\mathbb{Z}$. Then for $a,\;b\in\mathbb{Z}$, $D(a+b\sqrt{d})=b(\alpha+\beta\sqrt{d})$. For $x,\;y\in\mathbb{Z}$, we have
\begin{flalign*}
        &D(a+b\sqrt{d})\tau(x+y\sqrt{d})+\sigma(a+b\sqrt{d})D(x+y\sqrt{d})\\
        &=b(\alpha+\beta\sqrt{d})(x-y\sqrt{d})+(a+b\sqrt{d})y(\alpha+\beta\sqrt{d})\\
        &=(ay+bx)(\alpha+\beta\sqrt{d})=D\big((a+b\sqrt{d})(x+y\sqrt{d})\big).
       \end{flalign*}
       Thus $D$ is a $(\sigma,\tau)$-derivation of $\mathbb{Z}[\sqrt{d}]$. Now $\big(\sigma-\tau\big)(a+b\sqrt{d})=2b\sqrt{d}.$ Hence, $$D(a+b\sqrt{d})=bD(\sqrt{d})=\frac{D(\sqrt{d})}{2\sqrt{d}}\big(\sigma-\tau\big)(a+b\sqrt{d}).$$
       So if $\frac{D(\sqrt{d})}{2\sqrt{d}}\in\mathbb{Z}[\sqrt{d}]$, then $D$ is $(\sigma,\tau)$-inner. We have $\frac{D(\sqrt{d})}{2\sqrt{d}}=\frac{\beta}{2}+\frac{\alpha}{2d}\sqrt{d}$ and thus when $\alpha$ is divisible by $2d$ and $\beta$ is even, we get the required result.\\
       
       \item Here $d\equiv 1$ mod $4$. Again, by lemma \ref{O_k}, we know that $\sigma$ and $\tau$ are ring isomorphisms. Let $\sigma(\frac{1+\sqrt{d}}{2})=a+b\frac{1+\sqrt{d}}{2}$ for some $a,\;b\in\mathbb{Z}$. Now, $\sqrt{d}=2(\frac{1+\sqrt{d}}{2})-1$. Hence, $\sigma(\sqrt{d})=2\sigma(\frac{1+\sqrt{d}}{2})-1=(2a+b-1)+b\sqrt{d}.$ So\\
       
       \begin{flalign*}
        d&=\sigma(d)=\sigma(\sqrt{d})^2=\big((2a+b-1)+b\sqrt{d}\big)^2\\
        &=(2a+b-1)^2+db^2+2(2a+b-1)b\sqrt{d}.
       \end{flalign*}
       
       We get the following sets of equations:
       \begin{eqnarray}
        \label{eqn1}(2a+b-1)^2+db^2=&d\\
        \label{eqn2}2b(2a+b-1)=&0.
       \end{eqnarray}
       Equation \ref{eqn2} gives either $b=0$ or $2a+b-1=0$. If $b=0$, then equation \ref{eqn1} gives $a=\frac{1+\sqrt{d}}{2}$ which is a contradiction since $a\in\mathbb{Z}$. If $2a+b-1=0$ then equation \ref{eqn1} gives $b=\pm1$ and we get the following subcases: If $b=1$, then $a=0$. Thus $\sigma(\frac{1+\sqrt{d}}{2})=\frac{1+\sqrt{d}}{2}$, that is, $\sigma=id.$ If $b=-1$, then $a=1$. Then $\sigma(\frac{1+\sqrt{d}}{2})=\frac{1-\sqrt{d}}{2}$. In the same way, $\tau$ will also yield similar results. So we are again left with only two choices for the pair $(\sigma,\tau)$ as they are different ring isomorphisms. For $a,\;b\in\mathbb{Z}$, they are:
       
       \begin{enumerate}[$(a)$]
        \item $\sigma=id;\quad\tau\big(a+b\frac{1+\sqrt{d}}{2}\big)=a+b\frac{1-\sqrt{d}}{2}$,
        \item $\sigma\big(a+b\frac{1+\sqrt{d}}{2}\big)=a+b\frac{1-\sqrt{d}}{2};\quad\tau=id.$
       \end{enumerate}
Again, let us assume case $(a)$ as case $(b)$ will follow from symmetry. Let $D:\mathbb{Z}[\frac{1+\sqrt{d}}{2}]\longrightarrow\mathbb{Z}[\frac{1+\sqrt{d}}{2}]$ be a $\mathbb{Z}$-linear map such that $D(1)=0.$ Let $D(\frac{1+\sqrt{d}}{2})=\alpha+\beta\frac{1+\sqrt{d}}{2}$ for some $\alpha,\;\beta\in\mathbb{Z}.$ Then for $a,\;b\in\mathbb{Z}$, $D\big(a+b\frac{1+\sqrt{d}}{2}\big)=b(\alpha+\beta\frac{1+\sqrt{d}}{2}).$ For $x,\;y\in\mathbb{Z}$, we have

\begin{flalign*}
 &D\Big(a+b\frac{1+\sqrt{d}}{2}\Big)\tau\Big(x+y\frac{1+\sqrt{d}}{2}\Big)+\sigma\Big(a+b\frac{1+\sqrt{d}}{2}\Big)D\Big(x+y\frac{1+\sqrt{d}}{2}\Big)\\
 &=b\Big(\alpha+\beta\frac{1+\sqrt{d}}{2}\Big)\Big(x+y\frac{1-\sqrt{d}}{2}\Big)+\Big(a+b\frac{1+\sqrt{d}}{2}\Big)y\Big(\alpha+\beta\frac{1+\sqrt{d}}{2}\Big)\\
 &=(ay+bx+by)\Big(\alpha+\beta\frac{1+\sqrt{d}}{2}\Big)\\
 &=D\BIG(\Big(a+b\frac{1+\sqrt{d}}{2}\Big)\Big(x+y\frac{1+\sqrt{d}}{2}\Big)\BIG).
\end{flalign*}

Thus $D$ is a $(\sigma,\tau)$-derivation of $\mathbb{Z}[\frac{1+\sqrt{d}}{2}].$ Now $\BIG(\sigma-\tau\BIG)\big(a+b\frac{1+\sqrt{d}}{2}\big)=b\sqrt{d}.$ Hence
$$D\BIG(a+b\frac{1+\sqrt{d}}{2}\BIG)=bD\BIG(\frac{1+\sqrt{d}}{2}\BIG)=\frac{D\big(\frac{1+\sqrt{d}}{2}\big)}{\sqrt{d}}\BIG(\sigma-\tau\BIG)\BIG(a+b\frac{1+\sqrt{d}}{2}\BIG).$$

So if $\frac{D\big(\frac{1+\sqrt{d}}{2}\big)}{\sqrt{d}}\in\mathbb{Z}[\frac{1+\sqrt{d}}{2}]$, then $D$ is $(\sigma,\tau)$-inner. We have $\frac{D\big(\frac{1+\sqrt{d}}{2}\big)}{\sqrt{d}}=\frac{\beta}{2}d+(\alpha+\frac{\beta}{2})\sqrt{d}$. Thus, when $\beta$ is even the result follows.

                \end{enumerate}

   \end{proof}

\textbf{Acknowledgements:} The author would like to thank IISER Mohali for providing fellowship and good research facilities when this project was carried out.

%\begin{definition}\label{sigma-tau}A $(\sigma,\tau)$-derivation on $\mathcal{A}$ is an $R$-linear map $\delta$ satisfying
%\begin{equation*}
%\delta(ab)\;=\; \delta(a)\tau(b)+\sigma(a)\delta(b)
%\end{equation*}
%for $a,b\in \mathcal{A}$. The set of all $(\sigma,\tau)$-derivations on $\mathcal{A}$ is denoted by $\mathfrak{D}_{(\sigma,\tau)}(\mathcal{A})$. If $\sigma=\tau=id$, then $D$ is the usual derivation on $\mathcal{A}$.
%\end{definition}

%\begin{definition}\label{sigma}A $\sigma$-derivation on $\mathcal{A}$ is a $(\sigma, id)$-derivation, that is, an $R$-linear map $\delta$ satisfying
%\begin{equation*}
%\delta(ab)\;=\; \delta(a)b+\sigma(a)\delta(b)
%\end{equation*}
%for $a,b\in \mathcal{A}$. The set of all $(\sigma,\tau)$-derivations on $\mathcal{A}$ is denoted by $\mathfrak{D}_{\sigma}(\mathcal{A})$.
%\end{definition}

%%%%%%%%%%%%%%%%%%%%%%%%%%%% SOME USEFUL RESULTS %%%%%%%%%%%%%%%%%%%%%%%%%%%%%%%%%%%%%%%%%

 %%%%%%%%%%%%%%%%%% REFERENCES
\bibliographystyle{alpha}
\bibliography{OneBibToRuleThemAll}

\begin{thebibliography}{ELMS16}

\bibitem[AAC06]{AAH-06}
M.~Ashraf, S.~Ali, and C.Haetinger.
\newblock {On derivations in rings and their applications}.
\newblock {\em The Aligarh Bull. Math.}, 25:79--107, 2006.

\bibitem[Cha19]{Chau-19}
D.~Chaudhuri.
\newblock {$(\sigma,\tau)$-derivations of group rings}.
\newblock {\em Comm. Algebra}, 47(9):3800--3807, 2019.

\bibitem[ELMS16]{ELMS-16}
O.~Elchinger, K.~Lundengard, A.~Makhlouf, and S.D. Silvestrov.
\newblock {Brackets with $(\tau,\sigma)$-derivations and $(p,q)$-deformations
  of Witt and Virasoro algebras}.
\newblock {\em Forum Math.}, 28(4):657--673, 2016.

\bibitem[Goh10]{G-10}
A.~Gohr.
\newblock {On hom-algebras with surjective twisting}.
\newblock {\em J. Algebra}, 324:1483--1491, 2010.

\bibitem[Har02]{H-02}
J.T. Hartwig.
\newblock {\em {Generalized Derivations on Algebras and Highest Weight
  Representation of Virasoro Algebra}}.
\newblock Master's Thesis, Lund University,
  \url{https://sites.google.com/site/jonashartwig/xjobb1.pdf?attredirects=0 },
  2002.

\bibitem[HLS06]{HLS-06}
J.T. Hartwig, D.~Larsson, and S.D. Silvestrov.
\newblock {Deformations of Lie algebras using $\sigma$-derivations}.
\newblock {\em J. Algebra}, 295:314--361, 2006.

\bibitem[Jac56]{J}
N.~Jacobson.
\newblock {\em {Structure of Rings}}.
\newblock Colloquium Publications, Vol. 37, American Mathematical Society,
  Providence, R.I., 1956.

\bibitem[MM09]{MM-09}
M.S.~Moslehian M.~Mirzavaziri.
\newblock {A Kadison-Sakai Type Theorem}.
\newblock {\em Bull. Aust. Math. Soc.}, 79(2):249--257, 2009.

\bibitem[QJ08]{JL-08}
X.~Li Q.~Jin.
\newblock {Hom-Lie algebra structures on semi-simple Lie algebras}.
\newblock {\em J. Algebra}, 319:1398--1408, 2008.

\bibitem[She12]{S-12}
Y.~Sheng.
\newblock {Representations of Hom-Lie Algebras}.
\newblock {\em Algebr Represent Theor}, 15:1081--1098, 2012.

\bibitem[ST79]{ST}
I.~Stewart and D.~Tall.
\newblock {\em {Algebraic Number Theory}}.
\newblock Chapman and Hall, London, 1979.

\end{thebibliography}

\end{document}